\documentclass[reqno,11pt]{amsart}
\usepackage{amsmath, latexsym, amsfonts, stmaryrd, mathabx, amsthm, amscd, array, caption}
\usepackage[toc,page]{appendix}
\usepackage[utf8]{inputenc}
\usepackage{graphics,epsf,psfrag,epsfig}
\usepackage[colorlinks=true, pdfstartview=FitV, linkcolor=blue, citecolor=blue, urlcolor=blue,pagebackref=false]{hyperref}
\usepackage{bm}

\usepackage[notcite,notref,color]{}

\numberwithin{equation}{section}
\newtheorem{theorem}{Theorem}[section]
\newtheorem{lemma}[theorem]{Lemma}
\newtheorem{proposition}[theorem]{Proposition}
\newtheorem{cor}[theorem]{Corollary}
\newtheorem{rem}[theorem]{Remark}

\newtheorem{thrm}{Theorem}
\renewenvironment{proof}[1][Proof]{\begin{trivlist}
\item[\hskip \labelsep {\bfseries #1}]}{\qed\end{trivlist}}

\DeclareMathSymbol{\leqslant}{\mathalpha}{AMSa}{"36} 
\DeclareMathSymbol{\geqslant}{\mathalpha}{AMSa}{"3E} 
\DeclareMathSymbol{\eset}{\mathalpha}{AMSb}{"3F}     
\renewcommand{\leq}{\;\leqslant\;}                   
\renewcommand{\geq}{\;\geqslant\;}                   

\newcommand{\floor}[1]{\lfloor #1 \rfloor}
\newcommand{\ind}{\mathbf{1}}

\newcommand{\e}{{\rm e}}
\newcommand{\Ceff}{C_{{\rm eff}}}

\newcommand{\bbE}{{\ensuremath{\mathbb E}} }
\newcommand{\E}{{\ensuremath{\mathbb E}} }

\newcommand{\N}{{\ensuremath{\mathbb N}} }

\newcommand{\bbP}{{\ensuremath{\mathbb P}} }
\renewcommand{\P}{{\ensuremath{\mathbb P}} }

\newcommand{\R}{{\ensuremath{\mathbb R}} }

\newcommand{\bbZ}{{\ensuremath{\mathbb Z}} }
\newcommand{\Z}{{\ensuremath{\mathbb Z}} }

\newcommand{\ga}{\alpha}
\newcommand{\gb}{\beta}
\newcommand{\gd}{\delta}
\newcommand{\gep}{\varepsilon}       

\newcommand{\go}{\omega}

\newcommand{\gl}{\lambda}

\newcommand{\gamin}{\bar \alpha}

\makeatletter
\def\captionfont@{\footnotesize}
\def\captionheadfont@{\scshape}

\long\def\@makecaption#1#2{%
  \vspace{2mm}
  \setbox\@tempboxa\vbox{\color@setgroup
    \advance\hsize-6pc\noindent
    \captionfont@\captionheadfont@#1\@xp\@ifnotempty\@xp
        {\@cdr#2\@nil}{.\captionfont@\upshape\enspace#2}%
    \unskip\kern-6pc\par
    \global\setbox\@ne\lastbox\color@endgroup}%
  \ifhbox\@ne 
    \setbox\@ne\hbox{\unhbox\@ne\unskip\unskip\unpenalty\unkern}%
  \fi
  \ifdim\wd\@tempboxa=\z@ 
    \setbox\@ne\hbox to\columnwidth{\hss\kern-6pc\box\@ne\hss}%
  \else 
    \setbox\@ne\vbox{\unvbox\@tempboxa\parskip\z@skip
        \noindent\unhbox\@ne\advance\hsize-6pc\par}%
\fi
  \ifnum\@tempcnta<64 
    \addvspace\abovecaptionskip
    \moveright 3pc\box\@ne
  \else 
    \moveright 3pc\box\@ne
    \nobreak
    \vskip\belowcaptionskip
  \fi
\relax
}
\makeatother
\def\writefig#1 #2 #3 {\rlap{\kern #1 truecm
\raise #2 truecm \hbox{#3}}}

\author[Q.~Berger]{Quentin Berger}
\address{
LPMA, UMR 7599, Université Pierre et Marie Curie, case 188, 4 pl.~Jussieu, 75252 Paris Cedex 5, France \\
{\normalfont {\it email address:} \texttt{quentin.berger@upmc.fr}}}

\author[M.~Salvi]{Michele Salvi}

\address{CEREMADE, UMR 7534, Université Paris Dauphine, Place du Maréchal de Lattre de Tassigny, 75775 Paris Cedex 16, France\\
{\normalfont {\it email address:} \texttt{salvi@ceremade.dauphine.fr}}
}

\begin{document}

\title[Scaling of  1D Random Walks among biased Random Conductances]{Scaling of  sub-ballistic 1D Random Walks among biased Random Conductances}

\begin{abstract}
We consider two models of one-dimensional random walks among biased i.i.d.\ random conductances: the first is the classical exponential tilt of the conductances, while the second comes from the effect of adding an external field to a random walk on a point process (the bias depending on the distance between points). We study the case when the walk is transient to the right but sub-ballistic, and identify the correct scaling of the random walk: we find $\ga \in[0,1]$ such that $\log X_n /\log n \to \ga$. Interestingly, $\ga$ does not depend on the intensity of the bias in the first case, but it does in the second case. 
\medskip

\noindent  \emph{AMS  subject classification (2010 MSC)}: 
60K37, 
  60Fxx, 
  82D30.


\thanks{ M.~Salvi was financially supported  by the European Union's Horizon 2020 research and innovation program under the Marie Sk\l{}odowska-Curie Grant agreement No 656047.}
\end{abstract}

\maketitle

\section{Introduction}

We consider $(X_n)_{n\in\N}$, a discrete-time nearest-neighbor random walk in a random environment (RWRE) on $\Z$, whose transition probabilities are determined by a random collection $\go:=\{\go_i\}_{i\in\Z}$ of positive numbers in $[0,1]$ sampled according to 
a stationary and ergodic measure $\P$ (we will call $\bbE$ the relative expectation). For a fixed realization of the environment $\go$, $(X_n)_{n\in\N}$ starts at the origin and has transition probabilities
\begin{align*}
P^{\go} \big( X_{n+1} =i+1  \mid X_n= i\big) 
	=\go_i \,, 
	\qquad  P^{\go} \big( X_{n+1} =i-1  \mid X_n= i\big) =1-\go_i \, .
\end{align*}
Necessary and sufficient conditions on $\bbP$ for the random walk to be transient, and to have a positive asymptotic velocity in the transient case, are well known. Let us summarize here some of the results that can be found in Zeitouni, \cite[Theorems~2.1.2 and 2.1.9]{TZ04}: 

\begin{thrm}
\label{thm:transience}
Let $\rho_i:= \frac{1-\go_i}{\go_i}$ and assume that $\bbE[\log \rho_0]$ is well defined.
\begin{itemize}

	\item[(i)] The random walk $(X_n)_{n\in\N}$ is $\bbP$-a.s.: transient to the right if $\bbE[\log \rho_0 ]<0$; transient to the left if $\bbE[\log \rho_0] >0$;  recurrent if $\bbE[\log \rho_0]=0$.
\smallskip
	\item[(ii)] If $\bbE[\log \rho_0] <0$, then $\bbP$-a.s.\ the velocity $\mathbf{v} := \lim_{n\to\infty}  X_n/n$ exists $P^{\go}$-a.s.\ and is equal to $\bbE[\bar S]^{-1} \in[0,+\infty)$ with $\bar S := \frac{1}{\go_0} + \sum_{i=1}^{\infty} \frac{1}{\go_{-i}} \prod_{j=0}^{i-1} \rho_{-j} $.
\end{itemize}
\end{thrm}

One possibility for choosing the environment $\go$ is to sample with measure $\P$ a shift-ergodic sequence of positive random variables $\{c_{k}\}_{k\in \bbZ}$ attached to the edges of $\bbZ$ -- called \emph{conductances} -- and define for $i\in\Z$
\begin{equation}
\label{def:omega}
\go_i := \frac{c_{i}}{c_{i-1} + c_{i}} \, .
\end{equation}
We call the associated walk a random walk among random conductances (RWRC) and point out that $(X_n)_{n\in\N}$ is reversible with respect to the measure $\pi(k)=c_{k-1}+c_k$.
%
%
%
In this case, the $\rho_i$'s of Theorem \ref{thm:transience} are given by
$\rho_i=\frac{c_{i-1}}{c_{i}}$ and the formula for $\bar S$ simplifies to $\bar S = 1+ 2 \sum_{i=1}^{+\infty} \frac{c_{-i}}{c_0}$. It is not hard to see that the RWRC on ergodic conductances has always asymptotic speed $\mathbf{v}=0$.

\smallskip

In this paper we consider conductances sampled in an i.i.d.~manner, but we add an external force --~or bias~--, which ``pushes'' the random walk, say to the right. There are at least two natural choices for doing so, as discussed below. We will produce random walks $(X_n)_{n\in\N}$ that are transient  but have zero speed.
%
Our aim is to identify the correct scaling for $X_n$, \textit{i.e.}\ find $\ga\in[0,1]$ such that $X_n$ is typically of order $n^{\ga +o(1)}$. We stress that  the case of an RWRE with i.i.d.\ $\go_i$'s has been studied, for example in \cite{KKS75,ESZ09,ESTZ13}, and the picture is complete: the correct scaling has been identified as well as the scaling limits. However, in the case of i.i.d.\ conductances the $\go_i$'s are not i.i.d. anymore, and the phenomenology is actually very different, see  the discussion in Section \ref{sec:product}.

\smallskip

The first way of tilting the conductances (see Section \ref{subsection:birc}) is to multiply the $k$-th conductance by a factor $\e^{2k\gd}$, for some $\gd>0$. We call the related process random walk among Biased i.i.d.~Random Conductances (BiRC).
Somewhat surprisingly, the behavior of sub-ballistic BiRC has been studied only in dimension $d\geq 2$. In \cite{F13}, Fribergh shows that the walk is ballistic if and only if the expectation of the conductances is finite and then identifies the right order of rescaling in the sub-ballistic regime, depending on the tail of the distribution of the conductances at $+\infty$ (finer results are given in \cite{FK16}, under stronger assumptions). This represents one of the main differences with the one-dimensional model: as it will be clear from 
Theorem~\ref{thm:Biased1}, in our case both the integrability of the conductance \textit{and} the integrability of the inverse of the conductance play a role for determining the ballisticity of the walk and the right rescaling exponent. Roughly speaking, this is due to the fact that in higher dimension the walk will naturally go around traps generated by edges with a small conductance, whereas for $d=1$ it is not possible to avoid them.

\smallskip

The second way of adding a bias (see Section \ref{subsection:r1m}) is inspired from Physics: the present work was motivated by the study of the sub-ballistic regime for the Mott variable-range hopping (see \cite{FGS16,FGS17} and references therein), a model for the description of the movement of electrons in doped semi-conductors. 
The Mott walk is a long-range random walk on a random point process on $\R$ (we can see it as a walk on $\Z$ by projecting it); it may jump from its current position to any other site with a probability that decays exponentially in the distance and depends on some random energies associated to each point. One can then add an external electric field that induces a bias on the walk, but the bias depends this time on the distance between points. 
We consider here a simplification of this model, but we believe that our analysis captures the essence of the original one. We ignore the energies and allow the walk to jump only to its nearest-neighbors. We call it Range-$1$ Mott walk (R1M). We are not aware of any result for the sub-ballistic R1M in any dimension.
Interestingly, we find that the behavior of the R1M is very different from that of the BiRC: the scaling of the walk depends this time on the distribution of the conductances but \textit{also} on the intensity of the bias.
%
%
%
%

\subsection{Biased i.i.d.~Random Conductances (BiRC) with heavy tails}\label{subsection:birc}
We take an i.i.d.\ sequence $\{c_{k}\}_{k\in \bbZ}$ of random conductances under $\P$ and consider the \emph{biased} random walk among these conductances.
This corresponds to taking  conductances
\begin{equation}
\label{eq:biased_cond}
c_k^{\gd} = \e^{2\gd k } c_k  \, ,
\end{equation} 
where $\gd>0$ is the bias intensity.
Then, \eqref{def:omega} and the $\rho_k$'s appearing in Theorem \ref{thm:transience} become
\begin{align}\label{cucciolo1}
\go_k 
	= \frac{\e^{\gd} c_k}{ \e^{\gd} c_k + \e^{-\gd} c_{k-1}} 
	\qquad  \text{and } \qquad  
 	\rho_k  = \e^{-2\gd} \frac{c_{k-1}}{ c_{k}} \, .
\end{align}

We find that $\bbE[\log \rho_0] = - 2\gd <0$ so that 
Theorem~\ref{thm:transience} ensures that
$(X_n)_{n\in\N}$ is indeed transient to the right for any $\gd>0$ and the asymptotic velocity is
\begin{equation}
\mathbf{v}(\gd) = \frac{1}{\bbE[\bar S]} \qquad \text{with } \quad\bbE[\bar S] = 1 + 2 \bbE[c_0]\bbE[1/c_0] \frac{\e^{-2\gd}}{1-\e^{-2\gd}} \, .
\end{equation}
Hence, $\mathbf{v}(\gd)>0$ if and only if $\bbE[c_0] <+\infty$ and $\bbE[1/c_0] <+\infty$.
We realize that the zero velocity regime can occur for two different reasons: either $\bbE[c_0] = +\infty$, \textit{i.e.}\ $c_0$ has some heavy tail at $+\infty$; either $\bbE[1/c_0] = +\infty$, \textit{i.e.}\  $c_0$ has some heavy tail at $0$. 
This reflects the fact that a slowdown of the random walk is usually due to the presence of traps, \textit{i.e.}~regions of the space where it tends to spend the majority of the time, of two different kind (see Figure~\ref{figure1}). The first kind of trap  is produced by an edge $\{k,k+1\}$ with conductance $c_k=M\gg 0$ surrounded by two edges with conductances of $O(1)$: in this case the walk will tend to jump back and forth from point $k$ to $k+1$  several times (roughly a geometric number of times of mean $M$) before hopefully escaping to the right. The other kind of trap is represented by an edge $\{k,k+1\}$ with a conductance $\varepsilon\ll 1$: here the walk will attempt to jump to $k+1$ every time it finds itself on $k$, but it will manage to do it only with, roughly, probability~$\varepsilon$.

\begin{figure}[htb]
	\centering
	\setlength{\unitlength}{0.12\textwidth}  
	\begin{picture}(-0.3,1)(2,0.7)
		\put(-2,0){\includegraphics[scale=0.7]{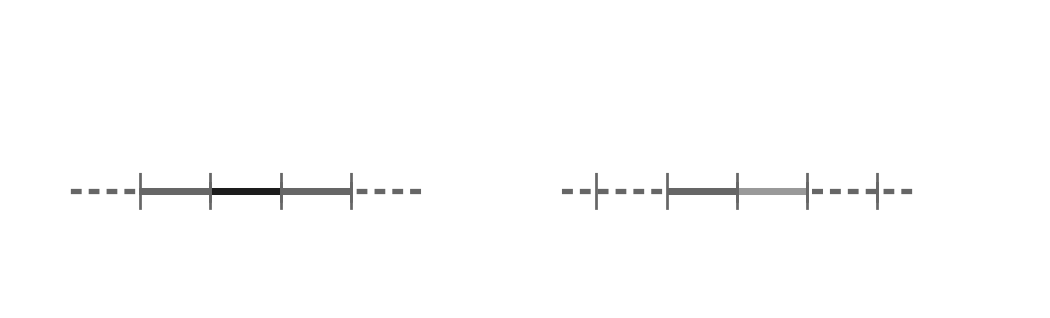}}
		\put(-1.15,0.65){\small $k-1\quad \;\;k\quad\,\, k+1\;\; k+2$}
 		\put(2.94,0.65){\small $k-1\quad \;\;k\quad\,\, k+1\;\; k+2$}
		\put(-0.92,1.4){ $\,O(1)\;\;\; M\quad\, O(1)$}
 		\put(-0.76,1.18){ $\downarrow\qquad\downarrow\qquad\downarrow$}
		\put(3.21,1.4){ $O(1)\;\;\;\; \varepsilon$}
		\put(3.33,1.18){ $\downarrow\qquad\downarrow$}
\end{picture}
\caption{\textit{On the l.h.s.~a trap of the BiRC due to the presence of a large conductance, $c_k=M\gg1$, surrounded by two conductances of $O(1)$. 
On the r.h.s.~a trap generated by a  small conductance $c_k=\varepsilon\ll 1$ preceded by a conductance of $O(1)$.}}
 \label{figure1}
\end{figure}
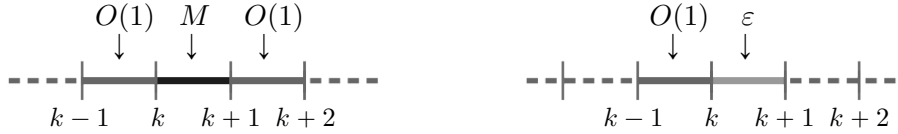

We therefore need to make some  assumption on the tail of $c_0$ at  $+\infty$ and $0$ to understand the asymptotic behavior of $(X_n)_{n\in\N}$: 
we suppose that there are some $\ga_\infty ,\ga_0\in[0,+\infty]$ such that
\begin{equation}
\label{hyp:tail1}
\lim_{M\to+\infty} \frac{\log \bbP(c_0 >M)}{\log M} = -\alpha_\infty \,, \qquad \lim_{\gep \to 0} \frac{\log \bbP(c_0<\gep)}{\log \gep} = \alpha_0 \, .
\end{equation}
We denote $\gamin := \min(\alpha_0,\alpha_{\infty})$, so that if the random walk is sub-ballistic then necessarily  $\gamin\in[0,1]$.

\begin{theorem}
\label{thm:Biased1}
Let $\gd>0$ and assume that \eqref{hyp:tail1} holds with $\gamin = \min(\alpha_0,\alpha_{\infty}) \leq 1$. Then we have
\[ \lim_{n\to\infty} \frac{\log X_n}{\log n} = \gamin  \qquad \bbP\otimes P^\go-a.s.\]
\end{theorem}

We stress that, in order to find a finer scaling for $X_n$, the assumption \eqref{hyp:tail1} on the conductances  needs to be strengthened. A natural way of doing so is to assume that there are slowly varying functions $L_\infty(\cdot)$ and $L_{0}(\cdot)$ such that
\begin{equation}
\label{regvar}
\bbP(c_0 >M) \stackrel{M\to\infty} \sim  L_{\infty}(M) M^{-\ga_{\infty}} \, , \quad \text{and} \quad \bbP\big(c_0<\gep \big)\stackrel{\gep\to0} \sim  L_{0}(\gep) \gep^{\ga_{0}}\, . 
\end{equation}
In \cite{FK16}, the authors consider the dimension $d\geq 2$, and use this assumption (for the tail at $+\infty$, since only large conductance traps exist in that case) to show that  the properly rescaled walk converges in law to the inverse of a stable subordinator. In dimension 1, the situation is fairly complicated due to the existence of several types of trap. When $\alpha_0\neq \alpha_\infty$ then the deepest traps are of one of the types presented in Figure \ref{figure1}. But when $\alpha_0 =\alpha_\infty$, a third kind of trap may be predominant (depending on the slowly varying functions $L_{\infty}$ and $L_{0}$), as a combination of two consecutive edges, one with high and the other with small conductance -- we refer to Section \ref{sec:product} for further discussion. This more subtle landscape is beyond the scope of this paper, and we postpone its study to a subsequent work.

\subsection{The  range-one Mott hopping model (R1M) with an external field}
\label{subsection:r1m}

%

We consider  a random sequence of positive random variables $(Z_i)_{i\in \bbZ}$ that are ergodic and shift invariant under  $\P$. The $(Z_i)_{i\in \bbZ}$ determine the position of points $\{x_i\}_{i\in\Z}$ on $\R$ such that $x_0:=0$, $x_k = \sum_{i=0}^{k-1} Z_i$ if $k\geq 1$ and $x_{k} = -\sum_{i=1}^{k} Z_{-i}$ if $k\leq -1$. The range-one Mott random walk $(\tilde Y_n)_{n\in\N}$ is the process starting in zero and jumping from point $x_k$ to point $x_{k+1}$ with probability equal to $\frac{\e^{- Z_k } }{ \e^{-Z_k } + \e^{- Z_{k-1} }}$ and to $x_{k-1}$ with the remaining probability.  We can also see this walk as the RWRC on the set of points $\{x_k\}$ with conductance associated to the edge $\{x_k,x_{k+1}\}$ equal to  $c_k:=\e^{-Z_k}$. 
We introduce an external electrical field of intensity $\lambda\in(0,1)$. From a physical point of view, this will result in a modification of the conductances which are given now by
\begin{equation}
\label{eq:Mott_cond}
c_k^{\lambda} := c_k\cdot \e^{ \lambda(x_k + x_{k+1})} \, .
\end{equation}
We will stick to the symbol $c_k^\gl$ when dealing with the conductances of the biased R1M, while we will use $c_k^\gd$ for the BiRC and $c_k^\star$ when dealing with both models at the same time. It will be convenient to look at a projection of $(\tilde Y_n)_{n\in\N}$ on $\Z$: we let  $(Y_n)_{n\geq 0}$  be the random walk on $\Z$ such that $\tilde Y_n := x_{Y_n}$. Studying the asymptotics of the two chains is equivalent, since in the transient case we have the relation $\frac{\tilde Y_n}{Y_n} \to \bbE[Z_0]$ as $n\to\infty$ by the ergodic theorem.

\smallskip

For $k\in\Z$, the probability of jumping from $k$ to $k+1$ for $(Y_n)_{n\in\N}$ and the $\rho_{k}$'s appearing in Theorem \ref{thm:transience} are, respectively,
\begin{align}\label{cucciolo2}
\go_k 
	= \frac{\e^{-(1-\lambda) Z_k } }{ \e^{-(1-\lambda) Z_k } + \e^{-(1+\lambda) Z_{k-1} }}  
	\qquad  \text{and } \qquad  
	\rho_k =  \e^{-(1+\lambda) Z_{k-1} +(1-\lambda) Z_k } \, .
\end{align}
%
%
%
%
%
It follows that $\bbE[\log \rho_0] = - 2 \lambda \bbE[Z_0] <0$,  so that part (i) of Theorem~\ref{thm:transience} guarantees that 
$(X_n)_{n\in\N}$ is transient to the right for any $\lambda>0$.
Part (ii) of the theorem and a straightforward computation also show that the limiting velocity is
\begin{equation}\label{calamaro}
\mathbf{v}(\lambda) = \frac{1}{\bbE[\bar S]},
\qquad \text{with } \quad
\bbE[\bar S] = 1+2 \bbE[\e^{(1-\lambda) Z_0}] \frac{ \bbE[\e^{-(1+\lambda) Z_0}]}{ 1- \bbE[\e^{-2\lambda Z_0}] }\,.
\end{equation}
Hence, $\mathbf{v}(\lambda)>0$ if and only if $\bbE[\e^{(1-\lambda) Z_0}] <+\infty$ (note that this condition is equivalent to the condition for ballisticity for the full-range Mott random walk, cfr.~ \cite[Theorem 1]{FGS16}). We define 
\begin{equation}
\label{def:lambdac}
\lambda_c:= \inf \big\{ \lambda  :\, \bbE[\e^{(1-\lambda) Z_0}] <+\infty \big\} \,, 
\end{equation}
so that $\mathbf{v}(\lambda) =0 $ if $\lambda <\lambda_c$ and $\mathbf{v}(\lambda) >0$ if $\lambda>\lambda_c$. 
Notice that from the definition of $\lambda_c$ it follows that 
$\bbP(Z_0 > t)  =\e^{-(1-\lambda_c) \,t\, (1+o(1))}$. Interestingly, by \eqref{calamaro} we see that $\mathbf{v}$ is continuous at $\lambda=\lambda_c$ if and only if $\bbE[\e^{(1-\lambda_c)Z_0 }] = +\infty$.

\begin{theorem}
\label{thm:Mott1}
For any $\lambda \leq \lambda_c$, define $\ga(\lambda) := \frac{1-\lambda_c}{1-\lambda}$. Then, 
$$  
\lim_{n\to\infty} \frac{\log Y_n}{\log n} 
	= \ga(\lambda) \qquad \bbP \otimes P^\go-a.s.
$$
\end{theorem}

We stress once more that a notable difference with Theorem \ref{thm:Biased1} is that here the scaling $\ga(\cdot)$ depends on the intensity of the bias $\gl$, whereas $\gamin$ does not depend on $\gd$.

\subsection{Discussion about the traps and products of heavy-tail random variables}
\label{sec:product}
In order to give a mathematical sense to the notion of traps, Sinai introduced in \cite{Sinai82} the potential function $V$, canonically associated to the environment $\go$. It can be defined as $V(k): = \sum_{i=1}^{k} \log \rho_i$ for $k\in\Z_+$ (for $k\in \Z_-$ add a minus sign, and $V(0)=0$)  and
it is a powerful and intuitive tool, since ``valleys'' in the potential landscape correspond to traps for the random walk. The notion of valleys as traps has then been extended in \cite{ESZ09} to the case of transient random walks  on $\mathbb{Z}$ (say to $+\infty$). Since in this case $\bbE[\log \rho_0]<0$, one has that the potential  has a negative drift ($V(k) \to -\infty$) and traps correspond to portions where $V$ increases by a large amount.
In particular, when the $\go_i$'s are i.i.d., the potential is a random walk with independent increments and the valleys are associated to its excursions: they are large regions with a large deviation behavior of the sum of $\log \rho_i$, see the discussion in Section~3 of \cite{ESZ09}.
This is in sharp contrast with our framework: in the BiRC (an analogous reasoning can be made for the R1M), the increments of $V$ are strongly correlated, and as a matter of fact we have $V(k) = \log c_0 -2\gd k - \log c_k$, for $k\ge 1$.   Valleys, \textit{i.e.}\ large increases in the potential, will then  be caused by isolated large values of $V(k) -V(k-1) = \log \rho_k$.
For us, it is therefore crucial to understand the tail of the distribution of $\rho_0 =e^{-2\gd} c_{-1}/c_0$.

\smallskip

The following general result goes in this direction and will be useful throughout the paper.
%
From now on, when treating the two models at the same time, we will use $\alpha$ for  $\gamin$ or $\ga(\gl)$ as defined in Theorems \ref{thm:Biased1} and \ref{thm:Mott1} respectively.

\begin{lemma}
\label{prop:producteasy}
Let $X,Y$ be two independent, positive, random variables.
Assume that there is some $\ga>0$ such that, as $t\to\infty$,
\[\bbP(X>t) = t^{-\ga +o(1)} \quad \text{and } \quad \bbP(Y >t) \leq t^{-\ga +o(1)} \, ,\]
\textit{i.e.}\ $Y$ has a lighter tail than $X$.
Then we have that
\[\bbP\big( XY >t \big) = t^{-\ga +o(1)} \quad \text{ as } t\to \infty \, .\]
\end{lemma}

\begin{proof}
For the lower bound, let $y_0 = \inf \{ y :\, \bbP(Y\leq y) \geq 1/2\}$; then 
\begin{align*}
  \bbP(XY > t)  \geq \bbP( Y \leq y_0 ) \bbP\big( X >  t/y_0 \big) = t^{-\ga +o(1)} \, .
\end{align*}

\noindent
For the upper bound, we fix some $\gep \in(0,\ga)$, and we write
\begin{align*}
\bbP(XY>t) 
	& \leq \bbP\big( Y > t  \big) + \bbE\big[ \bbP\big( X  >  t /Y  \mid Y \big)  \ind_{\{Y \leq t \}}\big] \\
	& \leq t^{-\ga +o(1)} +  c_{\gep} t^{-\ga+\gep}\bbE\big[ Y ^{ \ga -\gep} \ind_{\{ Y \leq t\}} \big]\, ,
\end{align*}
where we used the fact that there is a constant $c_{\gep}$ such that $\bbP(X > x) \leq c_{\gep} x^{-\ga +\gep}$ for all $x\geq 1$.
Then, because of our assumption on $Y$, we easily get that for all $\gep>0$, $\bbE[Y^{\ga-\gep}] <+\infty$.
Therefore, we get that there is a constant $C_{\gep}$ such that
\[\bbP(XY >t) \leq  t^{-\ga + o(1)} +  C_\gep t^{-\ga + \gep}\, ,\]
so that $\bbP( XY >t) \leq t^{-\ga +o(1)}$, since $\gep$ is arbitrary.
\end{proof}

Thanks to this lemma, we are able to obtain the tail of $\rho_0$:  we have that for both models
\begin{equation}
\label{eq:tailrho}
\bbP(\rho_0 > t) = t^{-\alpha +o(1)}\, .
\end{equation}
Indeed, for the BiRC we have that $\rho_0$ is the product of two independent random variables: $X:= e^{-2\gd} c_{-1}$, which has 
a tail $\bbP(X>t) = t^{-\ga_\infty+o(1)}$, and $Y:= 1/c_0$, which has 
a tail $\bbP(Y>t) = t^{-\ga_0+o(1)}$ (we might have to exchange the role of $X$ and $Y$ to properly apply Lemma~\ref{prop:producteasy}).
Analogously, for the R1M, $\rho_0$ is the product of 
$X:= \e^{(1-\lambda)Z_0}$ which has a tail $\bbP(X>t) = \bbP(Z_0 > \frac{1}{1-\lambda} \log t) = \exp\big\{ -  \frac{1-\lambda_c}{1-\lambda} \log t \,(1+o(1))\big\}=t^{-\ga(\gl)+o(1)}$, and $Y := \e^{-(1+\lambda) Z_{-1}}$ which is not larger than $1$ (and therefore satisfies $\bbP(Y>t) \leq t^{-\alpha(\gl) +o(1)}$).
Note that this proof gives an indication on the easiest way to create a large trap, \textit{i.e.}\ a large $\log \rho_i$: the tail of $\rho_0$ comes from having either a large $c_{-1}$ (with $c_0$ of order one), or a small $c_0$ (with $c_{-1}$ of order one), depending on which one is the easier.

\begin{rem}\rm
In the case where $c_0$ has a regularly varying tail at $0$ and $+\infty$, as in~\eqref{regvar}, \cite{EG80} gives that $\rho_0=c_{-1}/c_0$ has a regularly varying tail with exponent $\gamin = \min (\ga_0,\ga_{\infty})$, and  \cite[Corollary 5]{C86} provides the sharp behavior of $\bbP(\rho_0>t)$ as $t\to\infty$.
Let us stress here that in the case $\ga_0=\ga_{\infty}$, the main contribution to $\bbP(\rho_0>t)$ will come from a combination of having $c_{-1}$ large and $c_{0}$ small, therefore suggesting a new type of trap, different from those of Figure~\ref{figure1}. As an example, if $\bbP(c_0>t)\sim t^{-\gamin}$ and $\bbP(c_0<\gep) \sim \gep^{-\gamin}$ in \eqref{regvar}, a straightforward calculation gives that $\bbP\big( \rho_0 >t \big) \sim (\log t) t^{-\gamin}$, and the main contribution comes from all possibilities of having $c_{-1}\asymp t^a$ and $c_0  \asymp t^{a-1}$ with $a\in [0,1]$.
\end{rem}

\section{Proof of Theorems \ref{thm:Biased1} and \ref{thm:Mott1}}

\subsection{A fundamental preliminary result}

%

A central tool for the study of the walks are the hitting times
\begin{equation}
T_n := \inf \{ k \, , \, X_k=n\}\, .                   
\end{equation}
We will also use the notation $T_n$ for the hitting times of the R1M $(Y_n)_{n\in\N}$. Understanding the behavior of $T_n$ is the key to finding the right scaling for  $X_n$ (respectively, $Y_n$). In fact, the following proposition, in loose terms, shows that the increasing map $n\mapsto T_n$ is the inverse of the map $k\mapsto X_k$, up to an error of at most a constant times $\log n$. Corollary \ref{settimanaccia} shows the relation between the asymptotics of $T_n$ and that of the position of the walker.

\begin{proposition}\label{kingofthenorth}
There exists a constant $C>0$ such that, for $\P$--almost every $\go$ we have that, $P^\go$--a.s., there exists $n_0\in\N$ such that, for all $n\geq n_0$, 
\begin{align*}
X_{T_n+k}\geq n-C\,\log n\,,
\end{align*}
The same result holds for the R1M process $(Y_n)_{n\in\N}$.
\end{proposition}

\begin{proof}
We set $f(n):=n-C\,\log n$, with $C>0$ to be determined later on. We call $A_n:=\{W_k< f(n)\,,\;\mbox{for some } k \geq T_n\}$, where $W_k$ can be taken equal to $X_k$ or $Y_k$, and control 
\begin{align}\label{ponchielli}
P^\go(A_n)
	&\leq P^\go_n(\tau_{f(n)}<+\infty)
	=\lim_{M\to\infty} \frac{\Ceff(\{n\}\leftrightarrow\{f(n)\})}{\Ceff(\{n\}\leftrightarrow\{f(n)\}\cup \{M\})}
	\leq c_{f(n)}^\star \sum_{j=n}^\infty \frac{1}{c_j^\star}\,.
\end{align}
where $P^\go_n$ is the law of the random walk in random environment $\go$, starting from $n$.
Here $\tau_j$ is the first time the walk hits $j\in\N$, while $\Ceff(A\leftrightarrow B)$ indicates the effective conductance between set $A$ and set $B$, that is $\Ceff(A \leftrightarrow B):=\inf\{\sum_{k\in\Z}c_k(f(k+1)-f(k))^2\,,\,f|_{A}=0\,,\,f|_B=1\}$. For the equality in \eqref{ponchielli} we have used the well known formula for walks among conductances (see, \textit{e.g.}, \cite[Exercise 2.36]{LP16}) 
\begin{align}\label{useful}
P^\go_k(\tau_A<\tau_B)
	=\frac{C_{\rm eff}(\{k\} \leftrightarrow A)}{C_{\rm eff}(\{k\}\leftrightarrow A\cup B)}\,,
\end{align}
and then we have used the explicit expression of $\Ceff$ for conductances in series and in parallel (detailed later on in formulas \eqref{effettivamente1} and \eqref{effettivamente2}).

\smallskip

In the BiRC model case $c_j^\star=c_j^\delta$ and \eqref{ponchielli} becomes
\begin{align*}
P^\go(A_n)
	\leq \e^{-2\gd (n-f(n))}c_{f(n)}\sum_{j=0}^\infty \e^{-2\gd j}\frac{1}{c_{n+j}}
	=:\e^{-2\gd (n-f(n))}K_n(\omega)\,.
\end{align*}
We use Lemma~\ref{prop:producteasy} with $X=c_{f(n)}$ and $Y=\sum_{j=0}^\infty \e^{-2\gd j}/{c_{n+j}}$ (by Lemma~\ref{lem:tailXY} we have $\P(Y>t)= t^{-\ga_0+o(1)}$), or the other way around if $\ga_\infty\geq \ga_{0}$:
 we obtain $\P(K_n>t)= t^{-\gamin+o(1)}$. By a  Borel-Cantelli argument it follows that for any $\varepsilon >0$, $\P$--a.s., there exists an $m_0=m_0(\go)$ such that $K_m\leq m^{(1+\varepsilon)/\gamin}$  for each $m\geq m_0$. As a consequence, $\P$--a.s.~there exists $n_0$ such that, for each $n\geq n_0$, we have $P^\go(A_n)\leq n^{(1+\varepsilon)/\gamin}\e^{-2\gd (n-f(n))}=n^{(1+\varepsilon)/\gamin-2\gd C}$. Since this probability is summable for $C$ big enough, again by a Borel-Cantelli argument we have that $A_n$ happens only a finite number of times and the claim follows for $(X_n)_{n\in\N}$.
 
 \smallskip
 
For the R1M, one has $c_j^\star=c_j^\gl= \e^{-Z_j+ \lambda(x_j + x_{j+1})} $ in \eqref{ponchielli} and therefore
\begin{align*}
P^\go(A_n)
	\leq \e^{-Z_{f(n)}+\gl(x_{f(n)}+x_{f(n)+1})}\sum_{n=0}^\infty \frac{1}{\e^{-Z_j+\gl(x_j+x_{j+1})}}
	\leq  \sum_{j=n}^\infty \e^{Z_j-\gl (x_j-x_{f(n)})}.
\end{align*}
Now, $\P$--a.s.~there exists $n_0=n_0(\go)$ such that, for each $n\geq n_0$ and  
 $j\geq n$, we have $x_j-x_{f(n)}\geq (j-f(n))\E[Z_0]/2$ and also $Z_j\leq (j-f(n))\gl\E[Z_0]/4$. It follows that $P^{\go}(A_n)\leq n^{-C \gl\E[Z_0]/4}$, and for $C$ big enough this probability is summable. It follows that $A_n$ happens only a finite number of times and the claim follows for $(Y_n)_{n\in\N}$, too.
\end{proof}
An easy and important  consequence of Proposition \ref{kingofthenorth} is the following corollary, which says that in order to prove Theorems~\ref{thm:Biased1}-\ref{thm:Mott1}, we simply need to focus on $T_n$, and prove
\begin{equation}
\label{trenone}
\lim_{n\to\infty} \frac{\log T_n}{\log n} 
	=\frac{1}{\ga} \qquad \bbP\otimes P^{\go}-a.s.
\end{equation}
for both models (recall that $\alpha$ indicates either $\gamin$ or $\alpha(\gl)$).

\begin{cor}\label{settimanaccia}
For $\P$--almost every $\go$ we have that, $P^\go$--a.s.,
\begin{align}\label{panda1}
 \lim_{n\to\infty} \frac{\log T_n}{ \log n}  =  \frac1{a}
 \qquad
 \Longleftrightarrow 
\qquad
\lim_{n\to\infty} \frac{\log X_n}{ \log n} =a
\end{align}
with $a:\Omega\times(\Z^d)^{\otimes\N}\to\R_+$.
The same result holds for the R1M process $(Y_n)_{n\in\N}$.
\end{cor}

\begin{proof}
Throughout the proof we will use the fact that $T_n\to\infty$ almost surely as $n\to\infty$, which comes from the transience to the right of the walk. As in \cite[Lemma 2.1.17]{TZ04}, for $n\in\N$ we define $k_n$ to be the unique integer such that $T_{k_n}\leq n<T_{k_n+1}$ and note that $k_n\to\infty$. 
We first prove ``$\Rightarrow$''.
By Proposition \ref{kingofthenorth} we have that, for $n$ big enough, $ k_n-C\log k_n\leq X_n< k_n+1$. Hence, almost surely,
\begin{align*}
a=\liminf_{n\to\infty} \frac{\log(k_n-C\log k_n)}{\log T_{k_n+1}}
	&\leq\liminf_{n\to\infty} \frac{\log X_n}{ \log n}\\
	&\leq \limsup_{n\to\infty} \frac{\log X_n}{ \log n}
	\leq\limsup_{n\to\infty} \frac{\log (k_n+1)}{ \log T_{k_n}}
	=a\,.
\end{align*}
The direction ``$\Leftarrow$'' of \eqref{panda1} is easily proved by noticing that, almost surely,
\begin{align*}
\frac 1a
	=\liminf_{n\to\infty}\frac{\log T_n}{\log X_{T_n}}
	=\liminf_{n\to\infty}\frac{\log T_n}{\log n}
	\leq\limsup_{n\to\infty}\frac{\log T_n}{\log n}
	=\limsup_{n\to\infty}\frac{\log T_n}{\log X_{T_n}}
	=\frac 1a\,.\qedhere
\end{align*}
\end{proof}

\subsection{Lower bound in \eqref{trenone}}
Firstly we show that there are many traps deeper than $n^{1/\ga -\gep}$  between $0$ and $n$ with high probability; secondly we show how to use this fact to conclude that $\lim_{n\to\infty}\log T_n/\log n\geq 1/\ga$.

\smallskip
\noindent\textbf{Step 1:} {\it There are many deep traps.}
From the tail probability of $\rho_0$ given in \eqref{eq:tailrho} and the fact that the $\rho_k$ have only a range-two dependence, we get for both models the following lemma, as a standard application of the Borel-Cantelli lemma.
\begin{lemma}
\label{lem:maxBC}
For every $\gep>0$ we get that, $\bbP$-a.s., there exists some $n_{0}(\gep,\go)$ such that, for any $n \geq n_0$ the events
$$
A_n=A_n(\gep):=\Big\{{\rm Card}\big\{ 1\leq k\leq n :\ \rho_k > n^{1/\alpha -\gep}\big\} 
	\geq 2/\gep \Big\}\, 
$$
are verified.
\end{lemma}
\begin{proof}
We may reduce to a sequence of i.i.d.\ random variables $\rho_k$ by separately proving the statement for the sequences of i.i.d.\ r.v.'s $(\rho_{2i})_{i\in\N}$ and $(\rho_{2i+1})_{i\in\N}$. In the following, we therefore assume that $(\rho_k)_{k\geq 1}$ is a sequence of i.i.d.\ r.v.'s which satisfy \eqref{eq:tailrho}.
By the independence of the $\rho_k$ we have that 
\begin{align*}
\bbP(A_n^c)
	& \leq \bbP\Big( \exists \, 1\leq k_1,\ldots, k_{2/\gep}\leq n:\, \rho_{k} \le n^{1/\ga -\gep} \ \text{ for all } k \notin \{k_1,\ldots, k_{2/\gep}  \} \Big)\\
	&\leq  n^{2/\gep}   \bbP\big(\rho_0 \leq n^{1/\ga -\gep} \big)^{n-2/\gep}  	\leq n^{2/\gep} \exp \Big\{-\frac{n}{2}\, \bbP\big(\rho_0 > n^{1/\ga -\gep} \big) \Big\} \, ,
\end{align*}
the last inequality holding for $n \ge 4/\gep$, using also that $(1-x)^{n/2} \le \e^{-nx/2}$.
Then, we use \eqref{eq:tailrho} to get that $n\, \bbP(\rho_0 > n^{1/\ga -\gep} )= n^{-\ga \gep +o(1)} $. Therefore,
\[ 
\bbP( A_n^c ) 
	\leq \exp \big\{ - n^{-\ga \gep +o(1)}\big\}\, , \]
which is summable; the conclusion follows by Borel-Cantelli's lemma. 
\end{proof}

\smallskip

\noindent\textbf{Step 2:} {\it Deep traps slowdown the walk.}
We are now ready to prove the lower bound in \eqref{trenone}: we show that, for both models,
\begin{equation}
\label{eq:liminfTn}
 \liminf_{n\to\infty} \frac{\log T_n}{ \log n} 
 	\geq \frac1\ga \qquad \bbP \otimes P^{\go}-a.s.
\end{equation}

 Let us fix $\gep>0$ (we assume for simplicity that $2/\gep$ is an integer). Thanks to Lemma~\ref{lem:maxBC}, it will be sufficient to control $P^{\go}( T_n \leq  n^{1/\ga -2\gep} )$ only on the event $A_n$ that there are at least $2/\gep$ points $k$ such that $\rho_k \geq n^{1/\ga -\gep}$: let us denote them by $k_1,\ldots k_{2/\gep}$.
 On $A_n$ we have that
\begin{align*}
P^{\go} \big(  T_n \leq  n^{1/\ga - 2\gep} \big)
	\leq \prod_{i=1}^{2/\gep} P^{\go}\big( G_{k_i}  \leq  n^{1/\ga - 2\gep} \big)\,,
\end{align*}
where $G_k$ are geometric random variables with respective parameter $p_k = \go_k \leq 1/\rho_k$. In fact, every time the walk is in $k$, it tries to overjump the edge $\{k,k+1\}$ and has a probability $p_k$ of succeeding:
$G_k$ represents the random number of attempts the random walk has to make before crossing $\{k,k+1\}$ for the first time. 
Since we have that $p_{k_i} \leq 1/\rho_{k_i} < n^{-1/\ga +\gep}$ for all $i=1,..., 2/\gep$, we get that, provided that $n$ is large enough,
\[ 
P^{\go}\big( G_{k_i}  \leq  n^{1/\ga - 2\gep} \big)  
	= 1- \big(  1-p_{k_i} \big)^{n^{1/\ga - 2\gep}} 
	\leq 2 n^{-\gep} \, . 
\]
Therefore, on the event $A_n$ and for $n$ big enough, we get that 
\[
P^{\go}  \big(  T_n \leq  n^{1/\ga - 2\gep} \big) 
	\leq 4^{1/\gep} n^{-2} \, .
\]
Since by Lemma \ref{lem:maxBC} we know that $A_n$ is realized for $n$ big enough $\bbP$-a.s., an application of Borel-Cantelli's lemma gives that
\[\bbP-a.s.,\ P^{\go}-a.s.   \text{ one eventually has } T_n >  n^{1/\ga - 2\gep}\, . \]
This proves that for any $\gep>0$, a.s.\ $\log T_n /\log n \ge 1/\ga -2\gep$ for $n$ large enough, and \eqref{eq:liminfTn} follows.

\subsection{Upper bound in \eqref{trenone}}

Also for the upper bound we divide the proof in two steps. First we show that, $\bbP$-a.s., $E^{\go}[T_n] \leq n^{1/\ga +\gep} $ as $n\to\infty$; then we use this estimate to obtain a bound on $P^{\go}(T_n > n^{1/\ga +\gep/2})$ which in turn gives the upper bound in \eqref{trenone}.

\smallskip

\noindent\textbf{Step 1:} {\it Estimates on $E^{\go}[T_n]$.}
A central tool of our analysis will be the the following representation for the expectation of the hitting times (cfr.~\cite[Formula (3.22)]{B06})
\begin{align}\label{magicformula}
E^{\go} [T_n]=\frac{1}{C_{\rm eff}(\{0\}\leftrightarrow\{n\})}\sum_{k<n}\pi(k)P^\go_k(\tau_0<\tau_n)
\end{align}
where $\pi(k):=c_{k-1}^\gd+c_k^\gd$ is a reversible measure for $(X_n)_{n\in\N}$ (analogously, $\pi(k):=c_{k-1}^\gl+c_k^\gl$ for the R1M), $\tau_j$ denotes the first time the walk hits point $j\in\Z$ and $\Ceff$ is the effective conductance introduced after \eqref{ponchielli}.
We notice that in \eqref{magicformula} the quantity $P^\go_k(\tau_0<\tau_n)$ is equal to $1$ for $k\leq 0$, while for the other $k$'s we can use \eqref{useful}.
Moreover $C_{\rm eff}$ is very easy to handle in our case: for $i<k<j$ we have
\begin{align}
C_{\rm eff}(\{0\}\leftrightarrow\{k\})&=S(i,j-1)^{-1}\label{effettivamente1}\\
C_{\rm eff}(\{k\}\leftrightarrow\{i\}\cup\{j\})&=S(i,k-1)^{-1}+S(k,j-1)^{-1}\,,\label{effettivamente2}
\end{align}
where $S(i,j):=\sum_{\ell=i}^j\frac{1}{c_\ell^\star}$. For the first formula we have used that we have conductances in series, while for the second formula we have two sequences of conductances-in-series that are in parallel.
 Therefore, using \eqref{effettivamente1} and \eqref{effettivamente2}, we get that for both our models
\begin{align}\label{eq:ETn}
E^{\go} [T_n]
	&=\frac{1}{S(0,n-1)^{-1}}\Big(\sum_{k\leq 0}\pi(k)+\sum_{0<k<n}\pi(k)\frac{S(0,k-1)^{-1}}{S(0,k-1)^{-1}+S(k,n-1)^{-1}}\Big)\nonumber\\
	&=\sum_{k\leq 0}(c_{k-1}^\star+c_k^\star)S(0,n-1)+\sum_{0<k<n}(c_{k-1}^\star+c_k^\star)S(k,n-1) \,.
\end{align} 
This formula allows us to prove the following. 
\begin{proposition}
\label{prop:ETn}
For any $\gb >1/\ga$, we have that
\[\bbP\big(\, E^{\go}[T_n] \geq n^{\gb }\, \big)  \leq n^{1-\ga \gb +o(1)} \, .\]
As a consequence, for every $\gep>0$, we get that, $\bbP$-a.s., there exists some $n_0(\gep,\go)$ such that, for any $n\geq n_0$,
\begin{equation}
\label{ETn}
E^{\go}[T_n] \leq n^{ \frac{1}{\ga} +\gep+o(1)} \, .
\end{equation}
\end{proposition}

\begin{proof}
The second part of the lemma comes as an easy consequence of the first part. Indeed, we have $\bbP(E^{\go}[T_n] \geq n^{1/\ga +\gep}) \leq n^{-\ga \gep +o(1)}$, so an application of Borel-Cantelli lemma gives that, $\bbP$-a.s., $E^{\go}[T_{2^{k}}] \leq (2^k)^{1/\ga +\gep}$ for $k\geq k_0(\go)$. Then, since $E^{\go}[T_n]$ is non-decreasing, we can set $k_n=\floor{\log_2 n}$ and see that
\[
E^{\go}[T_n] 
	\leq E^{\go}[T_{2^{k_n+1}}] 
	\leq  (2^{k_n+1})^{1/\ga +\gep}
	\leq (2n)^{\frac1\ga +\gep}  
\]
for all $n\geq 2^{k_0(\go)}=:n_0(\go)$, 
 $\bbP$-a.s.

\smallskip

We therefore focus on the first part of the lemma.
We start from equation \eqref{eq:ETn}, which admits as an easy upper bound
\begin{align*}
E^{\go}[T_n] 
	& \leq  2  \sum_{k\leq 0}  c_k^\star  \sum_{\ell=0}^{n-1} \frac{1}{c_{\ell}^\star}  + 2 \sum_{ k =0}^{n-1} c_k^\star \sum_{\ell =k}^{n-1} \frac{1}{c_{\ell}^\star} 
	=: 2A+2B\,.
\end{align*}
Hence we get
\begin{equation}\label{canaglia}
\bbP\big(E^{\go}[T_n]  \geq n^{\gb}  \big) 
	\leq \bbP\big( A\geq n^\gb/4 \big) + \bbP\big( B \geq n^\gb /4  \big)\,,
\end{equation}
and we want to control the two probabilities. 
As a preliminary, we give an estimate on the tail of the two random variables of interest. We postpone the proof of the lemma to the end of Step 1.
\begin{lemma}
\label{lem:tailXY}
Set $X := \sum_{k\leq -1} c_{k}^\star$ and $Y:=  \sum_{\ell \geq 1} 1/c_{\ell}^\star$. We have:
\begin{itemize}
	\item[(i)] For the BiRC
	\begin{equation}
	P(X>t) 
		= t^{-\ga_{\infty} +o(1)}  \quad \text{and} \quad \bbP(Y>t) 
		= t^{-\ga_0 +o(1)}   \qquad \text{as } t\to\infty \, .
	\end{equation}
	\item[(ii)] For the R1M (recall that $\ga(\gl)=(1-\lambda_c)/(1-\lambda)$)
	\begin{equation}
	P(X>t) 
		\leq  t^{-\ga(\gl) +o(1)} \quad \text{and} \quad \bbP(Y>t) = t^{-\ga(\gl) +o(1)} \qquad \text{as } t\to\infty \, .
\end{equation}
\end{itemize}
\end{lemma}

\noindent
We can deal now with the first term (Term A) and second term (Term B) in the r.h.s.~of~\eqref{canaglia}.

\smallskip

\noindent{\it Term A.}
We pull out the term $\ell=0$ and $k=0$ in A, so that we can write
\[
A 
	=  \big(c_0^\star + X\big)\Big(\frac{1}{c_0^\star}+ \sum_{\ell =1}^{n-1} \frac{1}{c_\ell^\star} \Big)
	\leq
	1+\frac{1}{c_0^\star} X  + c_0^\star \,Y+ XY \, ,
\]
with $X$ and $Y$ defined as in Lemma \ref{lem:tailXY}.
Hence, we get that
\begin{align}\label{terma}
\bbP \big( A \geq  n^{\gb}/4 \big ) 
	\leq \bbP\big(  X/ c_0^\star\geq n^{\gb}/16 \big) + \bbP\big( {c_0^\star} Y \geq n^{\gb}/16 \big)+ \bbP\big(  X Y \geq n^{\gb}/16 \big) 	\, .
\end{align}
Using Lemma~\ref{prop:producteasy} (note that $X$, $Y$ and $c_0^\star$ are mutually independent), we therefore get that $\bbP(A\geq  n^{\gb}/4)  \leq n^{-\ga \gb + o(1)}$.

\smallskip

\noindent{\it Term B.}
For the term B, we  pull out the terms $\ell =k$, so that we can write
\[
B  \leq  n +\sum_{k=0}^{n-1} c_k^\star \sum_{\ell >k} \frac{1}{c_{\ell}^\star}\, .
\]
Hence, setting $V_k:= c_k^\star \sum_{\ell >k} \frac{1}{c_{\ell}^\star}$, and for $n$ so large that $ n^\gb/4 -n \geq n^\gb/5$, we get that
\begin{equation}
\label{eq:Bsum}
\bbP\big( B \geq n^\gb/ 4\big) 
	\leq \bbP\Big( \sum_{k=0}^{n-1} V_k \geq n^\gb/5 \Big)\, .
\end{equation}
Note that the $V_k$'s are not independent, but they have the same distribution as $V_0 = c_0^\star Y $.
Moreover, Lemma \ref{lem:tailXY}, combined with Lemma~\ref{prop:producteasy} (possibly exchanging the roles of $c_0^\star$ and $Y$) gives that 
\begin{equation}
\label{eq:tailW}
\bbP(V_0 >t) = \bbP\big( c_0^\star \,Y  > t\big) = t^{-\ga +o(1)} \quad \text{as } t\to \infty.
\end{equation}

Now, going back to \eqref{eq:Bsum}, we split the sum of the $V_k$'s into two parts, writing
\begin{equation}
\bbP\Big( \sum_{k=0}^{n-1} V_k \geq \frac{n^\gb}{5} \Big) 
	\leq \bbP\Big( \sum_{k=0}^{n-1} V_k \ind_{\{ V_k \ge  n^\gb/10  \}} 
	\geq \frac{ n^\gb}{10} \Big) + \bbP\Big( \sum_{k=0}^{n-1} V_k \ind_{\{ V_k <  n^\gb/10  \}}  \geq  \frac{n^\gb}{10} \Big)\, .
\label{eq:splitsum}
\end{equation}
For the first term, we notice that the event is realized if and only if one of the indicator function is non-zero. Using a union bound, the first term is therefore bounded by
\[
\bbP\Big( \exists \ k \in \{ 0,1,\dots, n-1\}:\, V_k \geq n^\gb/10 \Big) 
	\leq  n\, \bbP\big( V_0 \geq n^\gb/10 \big) 
	= n^{1- \ga \gb +o(1)} \, . 
\]
For the second term, we simply use Markov's inequality to get 
\begin{align*}
\bbP\Big( \sum_{k=0}^{n-1} V_k \ind_{\{ V_k <  n^\gb/10  \}}  \geq  \frac{n^\gb}{10} \Big)
	\leq10\, n^{1-\gb}\, \bbE\big[ V_0 \ind_{\{V_0 < n^\gb/10\}}\big] 
 	= n^{1-\ga \gb +o(1)}
\end{align*}
where we used that $\bbE[V_0 \ind_{\{V_0 <t\}}] = t^{1-\ga +o(1)}$ as $t\to\infty$, thanks to \eqref{eq:tailW}.

\smallskip

Plugging the last two estimates in \eqref{eq:splitsum} and going back to \eqref{eq:Bsum} proves that $\bbP(B \ge n^\gb/4) = n^{1-\ga\gb +o(1)}$. This and \eqref{terma}, inserted into \eqref{canaglia}, conclude the proof.
\end{proof}

\smallskip

\begin{proof}[Proof of Lemma \ref{lem:tailXY}]
For item $(i)$, we focus on $\bbP(X>t)$, since the tail of $Y$ can be found in a similar way. On the one hand, since $X\ge c_{-1}^{\gd}$, we simply have that
$\bbP(X>t)\geq \bbP(c_0>t)=t^{-\ga_\infty+o(1)}$.
On the other hand, $X = \sum_{k\le -1} c_{k}^{\gd} >t$ implies that there exists some $j\ge 1$ such that $c_{-j}^{\gd}> \gd \e^{-\gd j}t $, so that a union bound gives
\begin{align*}
\bbP(X>t)
	\leq\sum_{j=1}^\infty \bbP(c_0>\gd \e^{\gd j}t)
	&=\sum_{j=1}^\infty (\gd \e^{\gd j}t)^{-\ga_\infty+o(1)}
	=t^{-\alpha_\infty+o(1)}\, .
\end{align*}

\smallskip

For item $(ii)$, for $X$, we get that by Markov inequality and the independence of the $Z_i$'s
\begin{align*}
\bbP(X>t)
	\leq \bbP\Big(\sum_{k\leq -1}\e^{-\gl(Z_k+\dots+Z_{-1})}>t\Big)
	\leq \frac{1}{t}\sum_{k\geq 1}\bbE\big[\e^{-\gl Z_0}\big]^k
	\leq t^{-\ga(\gl)+o(1)}\,
\end{align*}
since $\bbE[\e^{-\gl Z_0}]$ is clearly smaller than $1$ and $\ga(\gl)<1$. 
We now turn to the tail of $Y$. First, we have the lower bound
\begin{align*}
\bbP(Y>t)
	&\geq \bbP \big(\e^{(1-\gl)Z_1-2\gl Z_2}>t \big)
	\geq \bbP(Z_2 \ge b ) \bbP(\e^{(1-\gl)Z_1}>t \e^{2\gl b})  
	= t^{-\ga(\gl)+o(1)}\,.
\end{align*}
For the first inequality we have restricted the sum in $Y$ to the first summand. For the secon inequality, we used the independence of the $Z_i$'s, and we used some constant $b>0$ such that $\bbP(Z_2\le b)>0$. Finally, for the last identity, we just used that for $a>0$, $\bbP( Z_1> a\log t) = t^{-a(1-\lambda_c) +o(1)}$, see \eqref{def:lambdac} and below.

For the upper bound, we take a constant $\eta>0$ to be determined later, and call $K:=\e^{-\eta}/(1-\e^{-\eta})$. We notice that $Y = \sum_{\ell\ge 1} 1/c_\ell^{\gl}>t$ implies that there exists at least some $\ell\geq 1$ such that $1/c_\ell^\gl>t\,\e^{-\eta \ell}/K$. A union bound therefore gives
\begin{align*}
\bbP(Y>t)
	\leq \sum_{\ell\geq 1}\bbP\big( \e^{(1-\gl)Z_\ell-2\gl(Z_0+\dots+Z_{\ell-1})}>t\,\e^{-\ell \eta}/K \big)\,.
\end{align*}
We fix some $\gep>0$ and we use Markov inequality (with the $(\alpha(\gl)\! - \!\gep)$-th moment) to obtain
\begin{align*}
\bbP(Y>t)
	&\leq \sum_{\ell\geq 1} \big(t^{-1} K\e^{\ell \eta}\big)^{\ga(\gl)-\gep}\,\bbE\Big[ \Big(\e^{(1-\gl)Z_\ell-2\gl(Z_0+\dots+Z_{\ell-1})} \Big)^{\ga(\gl)-\gep}\Big]\\
	&\leq t^{-\ga(\gl)+\gep}\,B\sum_{\ell\geq 1} \Big(\e^{ \eta(\ga(\gl)-\gep)}\,\bbE\Big[\e^{-2\gl(\ga(\gl)-\gep)Z_0}\Big]\Big)^\ell
\end{align*} 
where $B=\bbE[ \e^{(1-\gl_c-\gep(1-\gl)) Z_0}] K^{\ga(\gl)-\gep}$ is a finite constant for each $\gep>0$. The geometric sum is also finite if we choose $\eta$ small enough. It follows that $\bbP(Y>t)\leq t^{-\ga(\gl)+\gep+o(1)}$ for arbitrary $\gep>0$, and this concludes the proof.
\end{proof}

\smallskip

\noindent\textbf{{Step 2}: \it Conclusion of the argument.}
We finally prove that
\begin{equation}
\label{eq:limsupTn}
\limsup_{n\to\infty} \frac{\log T_n}{\log n} 
	\leq \frac1\ga \qquad \bbP \otimes P^{\go}-a.s.
\end{equation}
Let us fix $\gep>0$. Thanks to Markov inequality and Proposition~\ref{prop:ETn}-\eqref{ETn}, we get that, $\bbP$-a.s., for $n$ large enough, we have
\[
P^{\go}\big(T_n > (n/2)^{1/\ga +2\gep}\big) 
	\leq (n/2)^{-1/\ga -2\gep} E^{\go}[T_n] 
	\leq 2^{1/\ga +2\gep} n^{-\gep} \,  .
\]
Hence, by an application of Borel-Cantelli lemma, we get that, $P^{\go}$-a.s., there exists some $k_0=k_0(\go)\geq 1$ such that 
$T_{2^k} 
	\leq \big( 2^{k-1} \big)^{1/\ga +2\gep}$  for all  $k\ge k_0$.
We therefore conclude by observing that $T_n$ is increasing, so that for any $n $ such that $k_n: =\floor{\log_2 n} \ge k_0$ we get that 
\[T_n \le T_{2^{k_n +1}} \le \big( 2^{k_n} \big)^{1/\ga +2\gep} \le n^{\frac{1}{\ga} +2\gep}\quad \text{ for all } n\ge n_0(\go):=2^{k_0} \, .\]
This proves that for any $\gep>0$, one eventually has $\log T_n /\log n \le 1/\ga + 2\gep$ a.s., and \eqref{eq:limsupTn} follows.

\bibliography{bibliografia}
\bibliographystyle{alpha}
%
%
%
%
%
%
%
%
%
%

\end{document}